\documentclass[11pt]{article}
\usepackage{amsmath,amssymb,amsthm}
\usepackage{enumitem}
\usepackage[hidelinks]{hyperref}

\newtheorem{theorem}{Theorem}[section]
\newtheorem{lemma}[theorem]{Lemma}

\newtheorem{conjecture}[theorem]{Conjecture}
\newtheorem{claim}[theorem]{Claim}
\theoremstyle{definition}

\theoremstyle{remark}

\numberwithin{equation}{section}

\title{A counterexample to the Erd\H{o}s--S\'os bipartite-link conjecture}
\author{Tianchi Yang\thanks{School of Mathematics, Georgia Institute of
Technology, Atlanta, GA 30332, US. Email: \texttt{tyang439@gatech.edu}.}}
\date{\today}

\begin{document}

\maketitle

\begin{abstract}
Erd\H{o}s and S\'os conjectured that every $n$-vertex $3$-uniform
hypergraph whose link graphs are all bipartite has at most
$(1/4+o(1))\binom n3$ edges.
We disprove this conjecture by constructing counterexamples
with edge density at least $0.250000356>1/4$
for all sufficiently large $n$.
\end{abstract}
\section{Introduction}
A \emph{3-graph} is a 3-uniform hypergraph.
The \emph{link graph} of a 3-graph $\mathcal H$ at a vertex $v$,
denoted by $\partial_v\mathcal H$, is the graph on
$V(\mathcal H)\setminus\{v\}$ with an edge $xy$
whenever $\{v,x,y\}\in\mathcal H$.
We write $|\mathcal H|$ for the number of edges of $\mathcal H$.

Erd\H{o}s and S\'os conjectured that every $n$-vertex 3-graph
whose link graphs are all bipartite has fewer than $n^3/24$ edges.
The conjecture was recorded by Frankl and F\"uredi~\cite[p.~328]{FF84}.
The following asymptotic version appears
in~\cite[Section~2.9]{MPS11}.

\begin{conjecture}[Erd\H{o}s--S\'os]\label{conj:ES}
Every $n$-vertex $3$-graph $\mathcal H$ whose link graphs are all bipartite satisfies
\[
|\mathcal H|\le\left(\frac14+o(1)\right)\binom n3.
\]
\end{conjecture}

The triples of vertices that form directed triangles in a tournament
define a $3$-graph whose link graphs are all bipartite. In a uniformly random
tournament, each triple forms a directed triangle with probability
$1/4$, giving the conjectured leading constant~\cite{FF84}.
Other constructions attain the same asymptotic density, including
recursive blow-ups of a single $3$-edge; see~\cite{MPS11,FV13}.

Razborov~\cite{Raz10} obtained an upper bound of $0.266$ for the
asymptotic edge density using flag algebras.
Falgas-Ravry and Vaughan~\cite{FV13} improved this bound to $0.258295$.
More recently, Hou, Hou, Liu, Zeng, and Zhang~\cite{HouEtAl26}
studied F\"uredi's extension in which every vertex link is
$t$-partite. They proved that the asymptotic maximum edge density
is at most $1-t^{-1}-1/(12t^2)$ for every $t\ge2$.

We disprove Conjecture~\ref{conj:ES} with the following
family of $3$-graphs.

\begin{theorem}\label{thm:main}
For every positive integer $k$, there is a $3$-graph $\mathcal H_k$
on $n=497k$ vertices such that all its link graphs are bipartite and
\[
|\mathcal H_k| 
=\left(\frac1{24}+\frac{25}{420\,903\,336}\right)n^3
>\left(\frac14+0.000000356\right)\binom n3.
\]
\end{theorem}
By adding fewer than $497$ isolated vertices, we obtain counterexamples
for all sufficiently large $n$ with the same asymptotic edge density.

The construction in Section~2 starts with a nine-vertex
$3$-graph $\mathcal F$ whose link graphs are all bipartite
and proceeds as follows:
\[
\mathcal F\ \text{($9$ vertices)}
\longrightarrow
\mathcal H\ \text{($497$ vertices)}
\longrightarrow
\mathcal H_k\ \text{($497k$ vertices)}.
\]
We obtain $\mathcal H$ by taking an unequal blow-up of
$\mathcal F$ and adding the triples forming directed triangles
in a tournament inside each class.
The $3$-graph $\mathcal H_k$ is then the balanced $k$-fold
blow-up of $\mathcal H$.

\section{The construction}

Let $\mathcal F$ be the $3$-graph on $\{0,\ldots,8\}$ with the following $24$
edges, where $abc$ denotes $\{a,b,c\}$:
\begin{equation}\label{eq:template}
\begin{split}
\mathcal F=\{&014,017,026,028,034,058,067,\\
    &123,128,135,136,137,148,158,168,\\
    &234,256,278,345,346,347,348,578,678\}.
\end{split}
\end{equation}

\begin{lemma}\label{clm:template}
Every vertex link of $\mathcal F$ is bipartite.
\end{lemma}

\begin{proof}
For each $i\in V(\mathcal F)$, choose $A_i$ as follows and define
$B_i=V(\mathcal F)\setminus(\{i\}\cup A_i)$:
\[
\begin{array}{lll}
A_0=\{1,3,6,8\}, & A_1=\{0,3,8\}, & A_2=\{0,1,4,5,7\},\\[2pt]
A_3=\{0,2,5,6,7,8\}, & A_4=\{0,2,5,6,7,8\}, & A_5=\{0,1,2,4,7\},\\[2pt]
A_6=\{0,3,5,8\}, & A_7=\{0,3,8\}, & A_8=\{0,1,3,7\}.
\end{array}
\]
It follows directly from~\eqref{eq:template} that $A_i$ and $B_i$
form a bipartition of $\partial_i\mathcal F$.
\end{proof}

For a tournament $T_m$, let $\mathcal T_m$ be the $3$-graph
on $V(T_m)$ whose edges are the vertex triples of its directed
triangles. 
Each link $\partial_u\mathcal T_m$ is bipartite, with parts
$N^+_{T_m}(u)$ and $N^-_{T_m}(u)$,  the sets of  out-neighbours and
in-neighbours of $u$ in $T_m$, respectively.
The following edge counts are classical; see
Moon~\cite[Theorem~4 and its first corollary, pp.~10--11]{Moon68}.

\begin{claim}\label{clm:tournaments}
For every positive integer $m$, there is a tournament $T_m$
on $m$ vertices such that the associated $3$-graph $\mathcal T_m$ has
bipartite vertex links and
\[
|\mathcal T_m|=
\begin{cases}
(m^3-m)/24,&m\text{ odd},\\
(m^3-4m)/24,&m\text{ even}.
\end{cases}
\]
\end{claim}

Take nine disjoint classes $V_0,\ldots,V_8$ of sizes $m_0,\ldots,m_8$,
where
\[
(m_0,\ldots,m_8)=(21,7,35,162,160,7,21,35,49).
\]
For each $i$, place a copy of the tournament $T_{m_i}$ from
Claim~\ref{clm:tournaments} on $V_i$, and let $\mathcal T_{m_i}$ be the
$3$-graph formed by its directed triangles. Define $\mathcal H$ on
$V_0\cup\cdots\cup V_8$ by including exactly the following edges:
\begin{enumerate}
\item all triples with one vertex in each of $V_i,V_j,V_\ell$ whenever
      $\{i,j,\ell\}\in\mathcal F$;
\item all edges of $\mathcal T_{m_i}$ inside $V_i$, for each $i$.
\end{enumerate}
Thus $\mathcal H[V_i]=\mathcal T_{m_i}$ for each $i$, and no edge of
$\mathcal H$ meets exactly two classes.

\begin{lemma}\label{clm:finite}
The $3$-graph $\mathcal H$ has $497$ vertices, all its vertex links are bipartite,
and
\[
|\mathcal H|=5\,115\,152=\left(\frac1{24}+\frac{25}{420\,903\,336}\right)497^3.
\]
\end{lemma}

\begin{proof}
The class sizes give
\[
|V(\mathcal H)|=\sum_{i=0}^8m_i
=2\cdot21+2\cdot7+2\cdot35+162+160+49=497.
\] 
By construction and Claim~\ref{clm:tournaments},
\[
|\mathcal H|=\sum_{\{i,j,\ell\}\in \mathcal F}m_i m_j m_\ell
+\sum_{i=0}^8|\mathcal T_{m_i}|
=4\,758\,124+357\,028=5\,115\,152.
\]
The two sums are evaluated in Appendix~\ref{app:counts}.

Finally, fix $u\in V_i$, and let $A_i,B_i$ be the
bipartition of $\partial_i\mathcal F$ from Lemma~\ref{clm:template}.
We claim that
\[
N^+_{T_{m_i}}(u)\cup\bigcup_{j\in A_i}V_j,
\qquad
N^-_{T_{m_i}}(u)\cup\bigcup_{j\in B_i}V_j
\]
form a bipartition of $\partial_u\mathcal H$.
These sets partition $V(\mathcal H)\setminus\{u\}$.
Every hyperedge $\{u,v,w\}\in\mathcal H$ is of one of
the following two types:
\[
\text{(i)}\quad \{u,v,w\}\in\mathcal T_{m_i},
\qquad
\text{(ii)}\quad v\in V_j,\ w\in V_\ell,\ \{i,j,\ell\}\in\mathcal F.
\]
In case (i), $v,w$ belong to opposite sides
$N^+_{T_{m_i}}(u)$ and $N^-_{T_{m_i}}(u)$.
In case (ii), $j,\ell$ belong to opposite sides
$A_i$ and $B_i$.
Thus $\partial_u\mathcal H$ is bipartite.
\end{proof}

For a positive integer $k$, write $[k]=\{1,\ldots,k\}$.
Define $\mathcal H_k$ on $V(\mathcal H)\times[k]$ by
\[
\{(u,a),(v,b),(w,c)\}\in\mathcal H_k
\quad\Longleftrightarrow\quad
\{u,v,w\}\in\mathcal H,
\]
for $u,v,w\in V(\mathcal H)$ and $a,b,c\in[k]$.
Thus $\mathcal H_k$ is the balanced $k$-fold blow-up of $\mathcal H$.

\begin{lemma}\label{clm:blowup}
For every positive integer $k$, the $3$-graph $\mathcal H_k$ has $497k$ vertices
and bipartite vertex links, and $|\mathcal H_k|=k^3|\mathcal H|$.
\end{lemma}

\begin{proof}
By definition, $|V(\mathcal H_k)|=497k$ and
$|\mathcal H_k|=k^3|\mathcal H|$.
For each $(v,a)\in V(\mathcal H_k)$, the link
$\partial_{(v,a)}\mathcal H_k$ is a balanced $k$-fold blow-up
of $\partial_v\mathcal H$, together with $k-1$ isolated vertices.
It is therefore bipartite.
\end{proof}

\begin{proof}[Proof of Theorem~\ref{thm:main}]
Fix a positive integer $k$ and let $n=497k$. Lemmas~\ref{clm:finite}
and~\ref{clm:blowup} give an $n$-vertex $3$-graph $\mathcal H_k$ with bipartite
vertex links and
\[
|\mathcal H_k|
=k^3|\mathcal H|
=\left(\frac1{24}+\frac{25}{420\,903\,336}\right)n^3.
\]
Since $n^3>6\binom n3$, we obtain
\[
|\mathcal H_k|
>\left(\frac14+\frac{6\cdot25}{420\,903\,336}\right)\binom n3
>\left(\frac14+0.000000356\right)\binom n3.
\]
The asymptotic edge density is therefore strictly greater
than $1/4$, disproving Conjecture~\ref{conj:ES}.
\end{proof}

\section*{Declaration of AI use}

The author used GPT-6.1 to assist with exploring
proof strategies, checking proofs, and improving the exposition.
All mathematical arguments, results, and conclusions were
reviewed and verified by the author, who takes full responsibility
for the content of this paper.

\clearpage
\appendix
\section{Edge counts}\label{app:counts}

The class sizes are
$(m_0,\ldots,m_8)=(21,7,35,162,160,7,21,35,49)$.
The $24$ template edges in~\eqref{eq:template} contribute
\begin{align*}
\sum_{\{i,j,\ell\}\in\mathcal F}m_i m_j m_\ell
&=m_0m_1m_4+m_0m_1m_7+m_0m_2m_6+m_0m_2m_8+m_0m_3m_4+m_0m_5m_8\\
&\quad+m_0m_6m_7+m_1m_2m_3+m_1m_2m_8+m_1m_3m_5+m_1m_3m_6+m_1m_3m_7\\
&\quad+m_1m_4m_8+m_1m_5m_8+m_1m_6m_8+m_2m_3m_4+m_2m_5m_6+m_2m_7m_8\\
&\quad+m_3m_4m_5+m_3m_4m_6+m_3m_4m_7+m_3m_4m_8+m_5m_7m_8+m_6m_7m_8.
\end{align*}
Substituting the class sizes and evaluating all $24$ products gives
\[
\begin{aligned}
&\quad21\cdot7\cdot160+21\cdot7\cdot35+21\cdot35\cdot21+21\cdot35\cdot49+21\cdot162\cdot160+21\cdot7\cdot49\\
&\quad+21\cdot21\cdot35+7\cdot35\cdot162+7\cdot35\cdot49+7\cdot162\cdot7+7\cdot162\cdot21+7\cdot162\cdot35\\
&\quad+7\cdot160\cdot49+7\cdot7\cdot49+7\cdot21\cdot49+35\cdot162\cdot160+35\cdot7\cdot21+35\cdot35\cdot49\\
&\quad+162\cdot160\cdot7+162\cdot160\cdot21+162\cdot160\cdot35+162\cdot160\cdot49+7\cdot35\cdot49+21\cdot35\cdot49\\
&=23\,520+5\,145+15\,435+36\,015+544\,320+7\,203+15\,435+39\,690\\
&\quad+12\,005+7\,938+23\,814+39\,690+54\,880+2\,401+7\,203+907\,200\\
&\quad+5\,145+60\,025+181\,440+544\,320+907\,200+1\,270\,080+12\,005+36\,015\\
&=4\,758\,124.
\end{aligned}
\]

For the edges inside individual classes, Claim~\ref{clm:tournaments}
gives
\begin{align*}
\sum_{i=0}^8|\mathcal T_{m_i}|
&=\frac1{24}\bigl[(21^3-21)+(7^3-7)+(35^3-35)+(162^3-4\cdot162)\\
&\qquad+(160^3-4\cdot160)+(7^3-7)+(21^3-21)+(35^3-35)+(49^3-49)\bigr]\\
&=385+14+1\,785+177\,120+170\,640+14+385+1\,785+4\,900\\
&=357\,028.
\end{align*}

\end{document}